\documentclass[12pt,reqno]{amsart}

\usepackage{amssymb,mathrsfs,graphicx,color}
\usepackage{mathtools}
\usepackage{tikz}
\usepackage{comment}
\usetikzlibrary{shapes,arrows, decorations.pathreplacing,patterns}
\usepackage[margin=3cm]{geometry}
-\marginparwidth \oddsidemargin -9mm \evensidemargin \oddsidemargin
\usepackage{latexsym}
\usepackage{xcolor}
\usepackage{graphicx}

\def\eps{\varepsilon}
\def\R{\mathbb{R}}
\advance\hoffset by 5mm
\mathtoolsset{showonlyrefs=true}
\newtheorem{theorem}{Theorem}[section]
\newtheorem{corollary}{Corollary}[theorem]
\newtheorem{lemma}[theorem]{Lemma}
\newtheorem{prop}[theorem]{Proposition}
\newtheorem{remark}[theorem]{Remark}

\title{Chemotactic Reaction Enhancement in One Dimension}

\author{Yishu Gong}
\thanks{Department of
Mathematics, Duke University, Durham NC 27708, USA;
email: yishu.gong@math.duke.edu}
\author{Alexander Kiselev}
\thanks{Department of
Mathematics, Duke University, Durham NC 27708, USA;
email: kiselev@math.duke.edu}
\begin{document}

\begin{abstract}
Chemotaxis, the directional locomotion of cells towards a source of a chemical gradient, is an integral part of many biological processes - for example, bacteria motion, single-cell or multicellular organisms development, immune response, etc. Chemotaxis directs bacteria's movement to find food (e.g., glucose) by swimming toward the highest concentration of food molecules. In multicellular organisms, chemotaxis is critical to early development (e.g., movement of sperm towards the egg during fertilization). Chemotaxis also helps mobilize phagocytic and immune cells at sites of infection, tissue injury, and thus facilitates immune reactions. In this paper, we study a PDE system that describes such biological processes in one dimension, which may correspond to a thin channel, the setting relevant in many applications: for example,
spermatozoa progression to the ovum inside a Fallopian tube or immune response in a blood vessel.
Our objective is to obtain qualitatively precise estimates on how chemotaxis improves reaction efficiency,
when compared to purely diffusive situation. The techniques we use to achieve this goal include a variety of comparison principles and analysis of mass transport for a class of Fokker-Planck operators.
\end{abstract}
\maketitle

\section{Introduction}\label{introduction}
\hspace{\parindent}Chemotaxis is the movement of an organism in response to a chemical stimulus. Patlak \cite{Patlak} and Keller-Segel \cite{KS1} postulated equations to describe the density evolution of bacteria or mold exhibiting chemotaxis without metabolism in a concentration gradient of a given organic attractant. We state it in the simplified parabolic-elliptic form \cite{Pert}:
\begin{equation}\label{eq:keller}
    \partial_t \rho -\Delta \rho -\chi \nabla (\rho \nabla (-\Delta)^{-1} \rho)=0, \quad \rho(x,0)=\rho_0(x).
\end{equation}
Here $\rho(x, t)$ is the bacterial density as a function of $x$ and $t$, and $c(x,t)=(-\Delta)^{-1}\rho(x,t)$ is the attractant concentration. This relationship relies on the fact that the chemical is produced and diffuses at faster time scales than other dynamics in \eqref{eq:keller}. Keller-Segel equation has been extensively analyzed as a model for chemotaxis due to its ability to capture key phenomena and its intuitive nature; and therefore, many variants of the Keller-Segel model have been suggested to reflect more nuanced phenomena (see e.g.  \cite{DS},  \cite{HP}, \cite{JV1}, \cite{OH1}, \cite{PVW}). In particular, global regularity for the solutions to \eqref{eq:keller} in dimension one is well-known \cite{Pert}.

The relevance of chemotaxis is ubiquitous in biology. In many cases, chemotaxis serves to enable and enhance a variety of processes that can be thought of as reactions. For instance, bacteria use chemotaxis to move towards their food by swimming, sensing whether it has moved towards or away from the food, and finally reaching  and consuming the nutrient \cite{Sour}. During fertilization process, spermatozoa follow a concentration gradient of a chemoattractant secreted from the oocyte increasing the chances of fertilization \cite{Suarez}. In immune system, cytokines and chemokines   attract macrophages and neutrophils to the site of infection, ensuring that pathogens in the area are destroyed \cite{Turner}. Chemotaxis can also produce undesirable results:
in the case of cancer metastasis, organ-specific stromal cells release signaling proteins that induce chemotaxis and attract cancer cells \cite{Muller}.

In mathematical literature, existence, regularity, finite time blow up and long-time behavior of solutions to reaction diffusion equations with chemotaxis received much attention (see e.g. \cite{cpz,ESV,EW,MT,OTYM,TW,W1,W2,Winkler3,WJDE}).
On the other hand, as far as we know, relatively little rigorous work has been done to explain quantitatively and rigorously how chemotaxis influences reaction rates.
Motivated by the particular setting of fertilization process for corals, the authors of \cite{kr} and \cite{kr2} took the first steps toward quantitative understandings of reaction enhancement by chemotaxis by adding absorbing reaction and fluid flow to \eqref{eq:keller},
resulting in the single equation model
\begin{equation}\label{eq:coral}
    \partial_t \rho + (u\cdot \nabla)\rho -\Delta \rho +\chi \nabla (\rho \nabla (-\Delta)^{-1}\rho)=-\epsilon \rho^q, \quad \Delta \cdot u=0, \quad \rho(x,0)= \rho_0(x)\geq 0.
\end{equation}
Once a year, entire colonies of coral reefs simultaneously release their eggs and sperm, called gametes, into the ocean. The gametes, full of fatty substances called lipids, rise slowly to the ocean surface, where the process of fertilization takes place. Field measurements shows that corals have a fertilization success rate around 50\%, and often as high as 90\% (see e.g. \cite{lasker}, \cite{pennington}), while numeric simulations based on purely diffusive models \cite{ds} only predicted successful fertilization rate of less than 1\% as the gametes are strongly diluted. When fluid flow is added to the model as in \cite{chw} and \cite{ccw}, the gap between simulations and field measurements can be reduced, but does not appear to vanish completely. This may be due to the fact that eggs release a chemical that attracts the sperm (see e.g. \cite{coll1}, \cite{coll2}, \cite{miller}, \cite{miller2}
\cite{RZ}, \cite{ZR}). These actual and numerical experiments suggest that chemotaxis plays a role in coral and other marine animals fertilization.
In the framework of \eqref{eq:coral}, the papers \cite{kr} and \cite{kr2} showed that chemotaxis enhances reaction significantly - especially when reaction is weak (which is common in many biological processes \cite{VCCW}). The reaction rate can be measured through the decay of the total mass of the remaining density, $m(t)=\int \rho(x,t)dx$. On one hand, if chemotaxis is not present ($\chi=0$), the decay of $m(t)$ is very slow.
On the other hand, when chemotaxis is present ($\chi\neq0$), the decay of $m(t)$ and its time scale are independent of the reaction strength $\varepsilon$, and can be very fast and significant if the chemotactic coupling is strong.

The shortcoming of the model \eqref{eq:coral} is that it deals with a single equation, essentially assuming that the egg and sperm densities are identical and chemotactic on each other.
However, in most settings there are two distinct species only one of which is chemotactic on a chemical secreted by the other.
In the recent work \cite{KNRY}, the authors have considered a more realistic model given by
\begin{equation} \label{eq:full_system}
\begin{cases}
\partial_t\rho _{1}-\kappa\Delta \rho _{1}+\chi \nabla \cdot \left(\rho _{1}\nabla \left( -\Delta \right) ^{-1}\rho _{2}\right)=-\varepsilon \rho _{1}\rho _{2},\\
\partial_t\rho _{2}=-\varepsilon \rho _{1}\rho _{2}. \\
\rho_1(x,0)=\rho_{1,0}(x), \, \rho_2(x,0)=\rho_{2,0}(x), \, x\in \mathbb{R}.
\end{cases}
\end{equation}
Here the ambient flow is set to zero, and any essential effect from it is assumed to be encoded by effective diffusion.
Notice that we denote the strength of chemotactic coupling by $\chi,$ the same notation that is commonly used to denote the characteristic function of a set.
This is the often used notation in the field, and though we will overload $\chi$ by using it in these two different ways, it should be always clear from the context what we mean.
The density $\rho_2$ describes immobile target, while the density $\rho_1$ reacts with it and is guided by chemotactic attraction of elliptic Keller-Segel type.
Such model setting is reasonable in many instances - in case of the immune system signalling, some reproduction processes or even plant-pollinator interaction. The system \eqref{eq:full_system} models a fast diffusing chemical moving towards a fixed target under chemical
attraction created by the target, then reacts with the target when in contact. In \cite{KNRY}, the system \eqref{eq:full_system} is set in $\R^2.$ The main result is the comparison of the reaction half-time
with and without chemotaxis, for the case of the radial initial data (both for $\rho_1$ and $\rho_2$). The paper \cite{KNRY} also contains a detailed analysis of the linear problem: a Fokker-Planck operator with a logarithmic potential.

Our goal in this paper is to present the analysis of the one-dimensional analog. As we will see, at least in some configurations the chemotaxis has a significantly stronger effect in one dimension than in two. This may explain why in many instances
the biological systems feature quasi-one-dimensional structures where chemotaxis is involved. One example is spermatozoa progression to the ovum inside a Fallopian tube \cite{teves}, another is immune response in blood vessels \cite{manes}.
In this work, we also eliminate the radial assumption of \cite{KNRY}: our solutions do not have to be even.


Before we state our main result, we make further assumptions on the initial data $\rho_{2,0}$ and $\rho_{1,0}.$
For simplicity of statements, it is convenient for us to assume that $\rho_{2,0}$ is a characteristic function of the interval $[-\frac{\ell}{2}, \frac{\ell}{2}],$
but in this case lack of regularity leads to unnecessary technical issues. So we will assume that the initial data for $\rho_2$ is smooth, compactly supported, and very close
to the characteristic function $\chi_{[-\ell/2,\ell/2]}(x).$
Namely, we set
\begin{equation}\label{rho2init}
\rho_{2,0} = \sigma \eta(x), \,\,\,\chi_{[- \left(\frac12-\frac{1}{1000} \right) \ell, \left(\frac12-\frac{1}{1000} \right) \ell]}(x) \leq  \eta(x) \leq \chi_{[-\frac12 \ell, \frac12 \ell]}(x), \,\,\,\eta \in C_0^\infty.
\end{equation}
Here $\sigma$ is the parameter modulating the total mass of the $\rho_2$ species. It is not hard to adapt the argument for more general initial data $\rho_{2,0}$ by adjusting the constants in the arguments and results below.  
Regarding $\rho_{1,0},$ we would like to model the situation where the mass $M_0 = \int \rho_{1,0}(x,t)\,dx$ is very large and is situated at some distance $\sim L \geq \ell$ away from
the location of $\rho_2.$ Specifically, we will assume that
\begin{equation}\label{rho1init}
\int_{|x| \leq L} \rho_{1,0}(x)\,dx = M_0, \,\,\,\int_{|x| \leq L/2} \rho_{1,0}(x)\,dx \leq 1 \ll M_0.
\end{equation}

By scaling time and space, we can normalize two parameters of the system, and we choose to set $\kappa=\ell=1.$
This corresponds to the change of coordinates $x'=x/l,$ $t'=\kappa t/l^2.$ 
In the new coordinates, the system we will consider takes form
\begin{equation} \label{eq:sim_system}
\begin{cases}
\partial_t\rho _{1}-\Delta \rho _{1}+ \chi' \nabla \cdot \left(\rho _{1}\nabla \left( -\Delta \right) ^{-1}\rho _{2}\right)=-\varepsilon' \rho _{1}\rho _{2},\\
\partial_t\rho _{1}=-\varepsilon' \rho _{1}\rho _{2},\\
\rho_1(x,0)=\rho_{1,0}(x), \, \rho_2(x,0)=\rho_{2,0}(x), \, x\in \mathbb{R}.
\end{cases}
\end{equation}
where $\rho_{2,0} (x):=\sigma \eta(x)  \sim \sigma \chi_{[-\frac{1}{2},\frac{1}{2}]}.$ Here $\chi' = \ell^2 \chi/ \kappa,$ $\varepsilon' = \ell^2 \varepsilon / \kappa,$ and the parameters $M_0$ and $L$ from \eqref{rho1init}
are replaced by $M_0'=M_0/\ell^2$ and $L'=L/\ell$ respectively. In the rest of the paper, we will abuse notation and omit the primes in notation of the renormalized parameters. It will also be convenient for us to introduce
the parameter $\gamma = \sigma \chi' = \sigma \ell^2 \chi /\kappa.$
Our goal is to compare the two reaction scenarios, with chemotaxis and without chemotaxis. We can characterize the reaction rate through the decay rate of the $L^1$ norm of $\rho_2$ as density is non-negative.
Then we can measure the reaction rate through the time scale $T$ during which a fixed portion of the initial mass $\sim \sigma$ of $\rho_2$ will react. It will be convenient for us to take this portion to be $\frac{1}{4}$, and it is natural to call the corresponding time scale
the ``typical time scale" of the reaction.
The mass $\int \rho_1(\cdot, t) dx$ will not decay significantly over time $T$ as we will assume that there is more $\rho_1$ than $\rho_2$ initially ($M\gg\sigma$); this is a natural assumption for the processes we have in mind. Our first result is the following theorem.
\begin{theorem}\label{thm:main}
Assume that $\rho_1$ and $\rho_2$ solve \eqref{eq:sim_system} with the initial data as described in \eqref{rho1init} and \eqref{rho2init} (the latter with $\ell=1$).
Let \begin{equation}\label{assump319}
\chi^2 \sigma /\varepsilon \geq a >0.
\end{equation}
Then there exists $B>0,$ depending only on $a,$ such that if
\begin{equation}\label{assum1222}
\frac{M_0 \epsilon}{\gamma}, \gamma, \frac{M_0}{\sigma} \geq B,
\end{equation}
then
the typical reaction time $T_C$ during which at least a quarter of the initial mass of $\rho_2$ has been reacted out satisfies $T_C\lesssim\frac{L}{\gamma}$. On the other hand, in the purely diffusive case where $\chi=0,$
it takes a time $T_D\gtrsim \frac{L^2}{\log(M_0\varepsilon L)}$ to react out a quarter of the initial mass of $\rho_2$.
\end{theorem}

\begin{remark} Here we make the following observations and clarifications. \newline
1. The notation $\lesssim$, $\gtrsim$, and $\sim$ means, as usual, bounds with universal constants independent of the key parameters of the problem. \\
2. The assumption \eqref{assum1222} is natural in many settings. For instance, for many marine animals the spawning process involves a
typical number of sperm of the order $M_0 \sim 10^{10},$ number of eggs $\sim 10^6,$ and $\varepsilon \sim 10^{-2}.$ It is difficult to find data on measurements of chemotactic constant in the biological literature. \\
3. The condition involving $\chi^2 \sigma /\varepsilon \geq a >0$ appears from the application of Harnack inequality, which requires a bound on this constant from below. The constant $B$ only deteriorates if $a$ becomes small. \\
4. In the second condition in \eqref{rho1init} we could replace the region of integration by $|x| \leq 1$ as far as the upper bound on $T_C$ is concerned; but we need the support of $\rho_1$ to be initially at a distance $\sim L$
from the origin for the lower bound on $T_D.$
\end{remark}
Theorem \ref{thm:main} suggests in what kinds of situations the presence of chemotaxis can significantly improve reaction rates.
An important advantage that chemotaxis confers in dimension one compared to the result of \cite{KNRY} in dimension two is the reduction in the exponent of $L.$
This comes from the fact that in dimension two, the drift produced by chemotaxis has form $\sim -\gamma x/|x|^2$ at large distances from the origin, while in dimension one it is $\sim -\gamma x/|x|.$
We discuss it in more detail in Section \ref{Heuristics} below.



Naturally, the proof of Theorem \ref{thm:main} relies on the positive effect of chemotaxis speeding up the transportation of $\rho_1$ to the region where $\rho_2$ lies. In order to quantify this effect, we would like to compare the solution $\rho_1(x,t)$ of the system \eqref{eq:sim_system} to the solution of the Fokker-Planck equation
\begin{equation}\label{eq:fokker1}
    \partial_t \rho-\Delta \rho +\nabla\cdot(\rho\nabla H_0)=0
\end{equation}
where $\rho(x,0)=\rho_1(x,0)$ and $H_0(x,t) =\chi (-\Delta )^{-1}\rho_2(x,t)$. The time dependence of $H_0$ causes difficulty in analyzing \eqref{eq:fokker1}.
Our argument will run in two stages, and at the first stage we will consider the time interval on which at most a quarter of the total mass of $\rho_2$ is consumed.
On this time interval, we will replace $H_0$ by the stationary potential $H$ that creates the ``weakest" drift potential. 
It turns out that such $H$ will be given by
\begin{equation}\label{eq:def_H}
     H(x)=\begin{cases}
               \gamma (-\Delta)^{-1}\left( \chi _{\left[ -1/6, {1}/{2}\right] }(x)\right),\quad x\geq 0, \\
               \gamma (-\Delta)^{-1}\left( \chi _{\left[ -{1}/{2},1/6\right] }(x)\right),\quad x\leq 0,
            \end{cases}
\end{equation}
where, as before, $\chi_{\mathcal{S}}(x)$ denotes the characteristic function of a set $\mathcal{S}$. In one dimension, we have $H(x)=-\frac{1}{3}\gamma |x|$ for large $x$, so we are dealing with a Fokker-Planck equation with a linear potential.
We develop appropriate comparison tools to control the solution of \eqref{eq:sim_system} using \eqref{eq:fokker1} with potential \eqref{eq:def_H}, and derive transport estimates for \eqref{eq:fokker1} using a barrier function.
These transport estimates will tell us when a meaningful portion of the density $\rho_1$ has arrived into support of $\rho_2.$ We then show that by this time, at least a quarter of the mass of $\rho_2$ had to be consumed.

The paper is organized as follows. In Section 2 we provide the heuristic motivation of the main result, in Section 3 recall a statement of the theorem that gives global well-posedness for the system \eqref{eq:sim_system}.
In Section 4, we compare our solutions to \eqref{eq:sim_system} with solutions to the Fokker-Planck equations with time-independent linear potential. In Section 5, we give a transport estimates based on comparison principles.
In Section 6, we use the results in previous sections to prove Theorem \ref{thm:main} by estimating reaction that happens on ``pass through".

\section{The Diffusive Estimates and Heuristics}\label{Heuristics}
\subsection{Estimates in the Purely Diffusive Case}
\label{Heuristics_diffusive}
In this subsection we consider the following system without chemotaxis term:
\begin{equation} \label{eq:diffusive_system}
\begin{cases}
\partial_t\rho _{1}-\Delta \rho _{1}=-\varepsilon \rho _{1}\rho _{2},\\
\partial_t\rho _{2}=-\varepsilon \rho _{1}\rho _{2},\\
\rho_1(x,0)=\rho_{1,0}(x), \, \rho_2(x,0)=\rho_{2,0}(x), \, x\in \mathbb{R}.
\end{cases}
\end{equation}
where the initial data are the same as in \eqref{rho1init}, \eqref{rho2init}. 
Let $T_D$ denotes the time by which a quarter portion of the initial mass of $\rho_2$ is reacted out.
Our goal in this subsection is to find a rigorous lower bound for $T_D$.
This bound can be derived by comparing \eqref{eq:diffusive_system} with the following system:
\begin{equation} \label{eq:diffusive_system_comp}
\begin{cases}
\partial_t g _{1}=\Delta g _{1},\\
\partial_t g _{2}=-\varepsilon g _{1}g _{2},\\
g_1(x,0)=\rho_{1,0}(x), \, g_2(x,0)=\rho_{2,0}(x), \, x\in \mathbb{R}.
\end{cases}
\end{equation}
Comparison principle directly yields that $\rho_1(x,t)\leq g_1(x,t)$ for all $t\geq0$ and $x\in \mathbb{R}$, which implies that $\rho_2(x,t)\geq g_2(x,t)$.
Hence, if we denote $\tau$ the time it takes for $||g_2(\cdot, t)||_{L^1}$ to drop by a quarter, one would have $\tau\leq T_D.$

Recall that $g_1(\cdot,0)=\rho_1(\cdot,0)$ has mass $M_0$ and is concentrated distance $\sim L$ away from the origin. 
Observe that
\[ g_{1}\left( x,t\right) =\dfrac {1}{\sqrt {4\pi t}}\int _{\mathbb{R} }e^{-\frac {\left| x-y\right| ^{2}}{4t}}\rho _{1}\left( y,0\right) dy, \]
and therefore
\begin{equation*} \label{eq:g1}
\frac {1}{\sqrt {4\pi t}}e^{-\frac {C L^2}{t}}M_0 \leq  g_1(x,t) \leq \frac {1}{\sqrt {4\pi t}}e^{-\frac {C_2 L^2}{t}}M_0  \text{ for all $x\in [-\frac{1}{2},\frac{1}{2}]$}
\end{equation*}
for some $0<C_2<C,$ for all times.
One can substitute this in the equation for $g_2$ \eqref{eq:diffusive_system_comp} and get that
\begin{equation*} \label{eq:g2}
\partial_t\ln g_{2} \gtrsim -\dfrac{\varepsilon }{\sqrt {t}}e^{-\frac {C_2 L^{2}}{t}}M_0.
\end{equation*}
Hence, $\tau$ is determined by
\begin{equation*} \label{eq:half_time}
\int^{\tau }_{0}\frac {\varepsilon M_0}{\sqrt {t}}e^{-\frac {C_2 L^{2}}{t}}dt\gtrsim 1,
\end{equation*}
which, after the change of variable $y = C_2 L^2/t,$ is equivalent to
\begin{equation}\label{tvarchange}
L \int_{C_2L^2/\tau}^\infty \frac{e^{-y}}{y^{3/2}}\,dy \gtrsim \frac{1}{\eps M_0}.
\end{equation}
It is convenient to distinguish two asymptotic cases. \\
{\bf Case 1:} $\eps M_0 \gg 1.$ This is the case we are focusing on in this paper, as implied by \eqref{assum1222}. In this case we must have $C_2L^2/\tau \gg 1,$ and thus
\[ \int_{C_2L^2/\tau}^\infty \frac{e^{-y}}{y^{3/2}}\,dy \leq \int_{C_2L^2/\tau}^\infty e^{-y}\,dy \leq e^{-C_2L^2/\tau}. \]
Hence
\[ L e^{-C_2L^2/\tau} \gtrsim \frac{1}{\eps M_0}, \]
which leads to
\begin{equation}\label{tau1} \tau \gtrsim \frac{L^2}{\log(\eps L M_0)}. \end{equation}  
{\bf Case 2:} $\eps M_0 \ll 1.$ In this case we have $C_2L^2/\tau \ll 1,$ so that $\tau \gtrsim L^2.$ In addition, we estimate
\[ L\int_{C_2L^2/\tau}^\infty \frac{e^{-y}}{y^{3/2}}\,dy \leq L\int_{C_2L^2/\tau}^\infty \frac{1}{y^{3/2}}\,dy \leq \frac{2 \tau^{1/2}}{C_2^{1/2}}. \]
Therefore,
\begin{equation}\label{tau2} \tau \gtrsim \frac{1}{(\eps M_0)^2}. \end{equation}

The bound \eqref{tau1} directly implies the estimates on $T_D$ that appear in Theorem \ref{thm:main}.

\subsection{Formal Heuristics with Chemotaxis Term}
\label{Heuristics_chemotaxis}
Now we go back to the full system \eqref{eq:sim_system}, which includes the chemotaxis term. Let $T_C$ denote the time when a quarter of $||\rho_2||_{L^1}$ has been reacted out. The following formal argument suggest that adding this term may greatly 
reduce the typical reaction time in some parameter regimes.
Namely, in the regime $M\epsilon/\gamma\gg 1$, we formally show that $T_C\sim\frac{L}{\gamma}$, which can be much smaller compared to the purely diffusive typical reaction time scale $T_D\gtrsim \frac {L^2}{\log(\varepsilon LM_0)}$.

The heuristics for the estimate $T_C\sim\frac{L}{\gamma}$ is as follows.
Recall that in one dimension, the Green's function of the Laplacian is given by
\begin{equation}\label{invlap} (-\Delta)^{-1}f(x) = -\frac12 \int_{\R} |x-y| f(y)\,dy. \end{equation}
Therefore, due to the chemotaxis term, $\rho_1$ is subject to an advection velocity field
\begin{equation} \label{eq:velocity}
v \left( x,t\right) =\chi \nabla \left( \left( -\Delta \right) ^{-1}\rho _{2}\right) \left( x,t\right) =\dfrac {\chi}{2}\left( \int ^{\infty }_{x}\rho _{2}\left( y\right) dy-\int ^{x}_{-\infty }\rho _{2}\left( y\right) dy\right).
\end{equation}
By definition of $T_C,$ for any $t\leq T_C$ we have $||\rho_2(\cdot,t)||_{L^1} \geq \frac34 \sigma$, and $\rho_2(\cdot,t)$ is supported in $[-\frac{1}{2},\frac{1}{2}]$. Therefore, for all $|x|\geq \frac{1}{2}$, and $t\leq T$, we have the inward drift
\begin{equation}\label{eqn:v_field}
v\left( x,t\right) \cdot \frac{-x}{|x|}\sim\gamma.
\end{equation}
Recall that initially most of $\rho_1$ starts at a distance $\sim L$ from the origin. Therefore, within time $t\sim \frac{L}{\gamma}$, advection should bring a significant portion of $\rho_1$ into the neighborhood where $\rho_2$ is supported.
Once a significant portion of $\rho_1$ has been in contact with $\rho_2,$
since $\rho_1$ has mass $M_0$, ideally one may expect $\rho_1\sim M_0$ on the support of $\rho_2.$ 
This is not quite true as for large $\gamma$ the drift can in principle lead to very non-uniform distribution of $\rho_1$ on support of $\rho_2$ with $\sim e^\gamma$ factor difference between maximum and minimum. Nevertheless, 
heuristically, for now we ignore this issue. 
Then we expect that after $\sim L/\gamma$ transport stage, the mass of $\rho_2$ is going to decrease exponentially with rate $M_0\eps$. 
These arguments lead to the estimate
\[ T_C \lesssim 1+ \frac{L}{\gamma}+\frac{1}{M_0\eps};  \] observe that the second summand can be dropped if $M_0 \eps/\gamma \geq 1$ since we also have $L \geq 1.$ 

\section{Global regularity of the two-species system}\label{global}

In this section, we discuss global regularity of the system \eqref{eq:sim_system}. Let us recall that the Sobolev
space $H^s$ is defined as a set of functions for which the norm
\[\|f\|_{H^s}^2:= \int_{R^d}(1+ |k|^{2s})|\hat{f}(k)|^2\,dk\]
is finite. The global regularity for \eqref{eq:sim_system}
in any dimension has been proved in \cite{KNRY}, here we just recall the statement.

\begin{theorem}\label{globreg1120}
Suppose the initial conditions $\rho_{1,0}$, $\rho_{2,0}$ for
(\ref{eq:sim_system}) are non-negative, and lie in $L^1(\R^d) \cap H^m(\mathbb{R}^d)$ with an integer $m>d/2$,
then there is
a unique global in time solution 
\newline ~$(\rho_1(\cdot, t), \rho_2(\cdot,t)) \in C(L^1(\R^d) \cap H^m(\mathbb{R}^d),[0,\infty))$
to \eqref{eq:sim_system}.
\end{theorem}

\section{Mass Comparison Principle}

In this section, our goal is to compare $\rho_1$ with the solution $\rho$ to the Fokker-Planck equation:
\label{even_comparison}
\begin{equation} \label{eq:fokker_planck}
\partial_t\rho-\Delta \rho + \nabla \cdot \left(\rho \nabla  H\right)=0,
\end{equation}
where $H$ is an effective potential that will be chosen to be ``weaker" than the one created by
$(-\Delta)^{-1} \rho_2(x,t)$ for all times relevant for the argument.
It is convenient for us to work with potentials that are not smooth. In particular, our main potential $H$ is defined by
\begin{equation}
\label{eq:least_concentrated_H}
  H(x)=\begin{cases}
               -\frac{1}{3}x\gamma & \text{for }x\geq\frac{1}{2}, \\
               -\frac{1}{24} (3-4x+12x^2)\gamma & \text{for }0\leq x\leq\frac{1}{2},\\
               -\frac{1}{24} (3+4x+12x^2)\gamma & \text{for }-\frac{1}{2}\leq x\leq0,\\
               \frac{1}{3}x\gamma & \text{for }x\leq-\frac{1}{2}.
            \end{cases}
\end{equation}
On the positive real line, $H(x)$ equals $\gamma(-\Delta)^{-1}\chi_{[-\frac16,\frac{1}{2}]}$, while on the negative real line 
\newline 
 $\gamma(-\Delta)^{-1} \chi_{[-\frac{1}{2},\frac16]}$.
\begin{figure}
\centering
\includegraphics[width=0.5\linewidth]{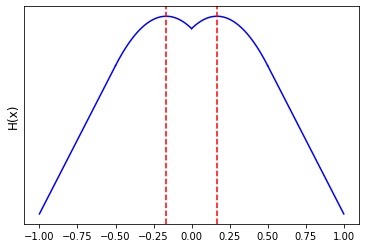}
  \caption{Weakest potential}
  \label{fig:H(X)}
\end{figure}
Note that $H$ has a discontinuous derivative at $x=0$ and discontinuous second derivative at $x = \pm 1/2.$
Standard parabolic regularity (see e.g. \cite{KrylSob}) implies that the solution $\rho(x,t)$ of \eqref{eq:fokker_planck} with such potential is smooth
away from $x=0,\pm 1/2.$ It is not difficult to show that $\rho(x,t)$ for $t >0$ will have a jump in the second derivative at $x =\pm 1/2$ and a jump in the first derivative
at $x=0.$ This regularity is sufficient for all manipulations carried out below.

To begin, we state a mass comparison principle that was proved in \cite{KNRY} in the two dimensional case. We provide a sketch of the argument for the sake of completeness.
Let $r=|x|$ and note that $\partial_r f(x) = \partial_x f(x) {\rm sgn}(x)$ in one dimension.  
\begin{prop}\label{prop:mass_comp}
Suppose that $u_1$ and $u_2$ are non-negative solutions to
\begin{equation}
    \partial_t u_i-\Delta u_i+\nabla\cdot (u_i\nabla H_i)=0\quad\text{for $i=1,2$},
\end{equation}
where $H_2$ is even with $\partial_r H_2(x,t)\geq \partial_r H_1(x,t)$ for all $x$ and $t \geq 0$, and $u_1$ is more concentrated than $u_2$ at $t=0$, i.e,
\begin{equation}
    \int_{-r}^r u_1(x,0)dx\geq \int_{-r}^ru_2(x,0)dx\text{ for all }r\geq 0.
\end{equation}
Then $u_1(\cdot,t)$ is more concentrated than $u_2(\cdot,t)$ for all $t\geq0$, i.e,
\begin{equation}
    \int_{-r}^r u_1(x,t)dx\geq \int_{-r}^ru_2(x,t)dx\text{ for all }t\geq 0.
\end{equation}
\end{prop}
Note that $u_{1,2}$ and $H_1$ are not required to be even, only $H_2$. Also, $H_{1,2}$ can both be time dependent, even though in our case here 
$H_2$ will be stationary. 
\begin{proof}
Define the masses
\begin{equation}
    M_{i}(r,t):=\int_{-r}^r u_i(x,t)dx.
\end{equation}
Then we have
\begin{align}
\begin{split}
    \partial_tM_i(r,t)&=\int_{-r}^r \Delta u_i dx-\int_{-r}^r\nabla\cdot (u_i\nabla H_i)dx \\&= \partial_r u_i(r)-\partial_r u_i(-r)-(u_i(r)\partial_r H_i(r)+u_i(-r)\partial_r H_i(-r)) \\
    &=\partial_{r}^2M_i(r,t)-(u_i(r)\partial_r H_i(r)+u_i(-r)\partial_r H_i(-r)).
\end{split}
\end{align}
Using the even symmetry of $H_2,$ $\partial_r M_i \geq 0,$ and the fact that $\partial_r H_2\geq \partial_r H_1$, we find that
\begin{align}
\begin{split}
        \partial_t (M_1-M_2)-\partial_{r}^2(M_1-M_2)\geq \partial_rH_2(r,t)\partial_rM_2-\max\{\partial_rH_1(r,t),\partial_rH_1(-r,t)\}\partial_rM_1 \\
        \geq (-\partial_rH_2(r,t))\partial_r(M_1-M_2)+(\partial_rH_2(r,t)-\max\{\partial_rH_1(r,t),\partial_rH_1(-r,t)\})\partial_rM_1\\
        \geq (-\partial_rH_2(r,t))\partial_r(M_1-M_2).
\end{split}
\end{align}
The standard parabolic comparison principle implies that
\begin{equation}
    M_1(r,t)\geq M_2(r,t) \quad \text{for all } r,t\geq 0.
\end{equation}
\end{proof}

Here is the reason for our choice of the potential $H$ in \eqref{eq:least_concentrated_H}.
\begin{prop}\label{comppot}
Let $f(x)\geq0$ be a function 
such that
\begin{equation}\label{fassump}
    0\leq f(x)\leq \rho_2(x,0) \quad \text{and} \quad ||f||_{L^1}\geq \frac{3}{4}||\rho_2(\cdot,0)||_{L^1}.
\end{equation}
then $(-\Delta)^{-1}f$ will create a stronger attraction towards the origin than $H:$
namely,
\begin{equation}\label{bal1229} \partial_r H(x) \geq \partial_r (-\Delta)^{-1}f(x) \end{equation}
for all $|x|>0.$
\end{prop}
\begin{proof}
From \eqref{eq:velocity}, we can see that the drift $\partial_r (-\Delta)^{-1}f(x)$ is determined by the balance between mass of $\rho_2$ on the left and right of $x$.
If $x \geq \frac12,$ then \eqref{bal1229} holds simply because $\|f\|_{L^1} \geq \sigma\|\chi_{[-1/6,1/2]}\|_{L^1}$ due to \eqref{rho2init} and \eqref{fassump}.
If $0 \leq x \leq \frac12,$ then \[ \sigma \int_x^{\infty}\chi_{[-1/6,1/2]}(y)\, dy \geq \int_x^{\infty} f(y)\,dy \]
due to \eqref{rho2init}. On the other hand,
\[ \int_{-\infty}^x f(y)\,dy \geq \sigma \int_{-\infty}^x \chi_{[-1/6,1/2]}(y)\, dy \]
due to \eqref{fassump} and \eqref{rho2init}.
Consideration of $x<0$ is similar.
\end{proof}
\begin{prop}\label{thm:FKvsChemotaxis}
Let $\rho_1(x,t),\rho_2(x,t)$ solve the system \eqref{eq:sim_system} with $\rho_{1,0}(x)$ and $\rho_{2,0}(x)$ as in \eqref{rho1init}, \eqref{rho2init}. 
Denote by $T_C$ the time it takes for the $L^1$ norm of $\rho_2$ to drop by a quarter.
Let $\rho(x,t)$ solve the Fokker-Planck equation \eqref{eq:fokker_planck} with the drift potential $H$ given by \eqref{eq:least_concentrated_H}
and the same initial data as $\rho_1$. Then for every $0 <t\leq T_C$ and every $r>0$ 
we have
\begin{equation} \label{eq:FKvsChemotaxis}
\int^{r}_{-r}\rho _{1}(x,t)dx\geq\int^{r}_{-r}\rho _{}(x,t)dx-\frac{\sigma}{4}.
\end{equation}
\end{prop}
\begin{proof}
Let $\widetilde{\rho}_1$ solve the equation for $\rho_1$ without reaction term:
\begin{equation} \label{eq:modified_chemotaxis}
\partial_t\widetilde{\rho }_1-\Delta \widetilde{\rho }_1+ \nabla \cdot \left(\widetilde{\rho }_1\nabla \left( -\Delta \right) ^{-1}\rho _{2}\right)=0,
\end{equation}
and let $\widetilde{\rho}_1$ have the same initial data as $\rho_1$. Note that since \[ \int_{\R} \rho_2(x,t)\,dx \geq \frac34 \int_{\R} \rho_{2,0}(x)\,dx \] for all $t \leq T_C,$
$\partial_r (-\Delta)^{-1}\rho_2(\cdot,t)\leq \partial_r H(\cdot)\leq 0$ for all such $t$ by Proposition \ref{comppot} (where as usual $H$ is given by \eqref{eq:least_concentrated_H}).
Thus Proposition \ref{prop:mass_comp} gives
\begin{equation}\label{eq:mass_rho_rho_tilde}
    \int_{-r}^r\Tilde{\rho}(x,t)dx\geq \int_{-r}^r\rho(x,t)dx\quad \text{for all } t\leq T_C \text{ and } r\in\left(0,\infty\right).
\end{equation}
To prove \eqref{eq:FKvsChemotaxis}, it suffices to compare $\rho_1$ and $\widetilde{\rho}_1$ and show that
\begin{equation} \label{eq:FKvsChemotaxis_2}
\int^{r}_{-r}\rho _{1}(x,t)dx\geq\int^{r}_{-r}\widetilde{\rho }_{1}(x,t)dx-\frac{\sigma}{4} \hspace{5mm}\text{ for }t\leq T.
\end{equation}
This is due to the following simple observation: on one hand, since the only difference between the equations for $\rho_1$ and $\widetilde{\rho}_1$ is the reaction term, applying parabolic comparison principle on these two densities directly gives
\begin{equation} \label{eq:FKvsChemotaxis_3}
\widetilde{\rho} _{1}(x,t)\geq\rho _{1}(x,t) \hspace{5mm}\text{ for all }x,t.
\end{equation}
On the other hand, we claim that for all $t\leq T_C$, the mass of $\rho_1$ and $\widetilde{\rho}_1$ does not differ by more than $\frac{\sigma}{4}$. Note that the mass of $\widetilde{\rho}_1$ is preserved over time, while the mass for $\rho_1$ is decreasing due to the reaction term. To prove the claim, it suffices to show that the mass of $\rho_1$ can at most drop $\frac{\sigma}{4}$ over time $T_C$. This is true since the equations for $\rho_1$ and $\rho_2$ have the same reaction term, hence
\begin{equation} \label{eq:FKvsChemotaxis_4}
\left\| \rho _{1}\left( \cdot ,0\right) \right\| _{L^{1}}-\left\| \rho _{1}\left( \cdot ,t\right) \right\| _{L^{1}}=\left\| \rho _{2}\left( \cdot ,0\right) \right\| _{L^{1}}-\left\| \rho _{2}\left( \cdot,t\right) \right\| _{L^{1}}\leq\frac{\sigma}{4}
\hspace{5mm}\text{ for }t\leq T_C,
\end{equation}
where in the last step we used the definition that $T_C$ is the time for a quarter of $\rho_2$ to be reacted out. Combining \eqref{eq:FKvsChemotaxis_3} with \eqref{eq:FKvsChemotaxis_4} yields
\begin{equation*} \label{eq:FKvsChemotaxis_5}
\begin{split}
\int^{r}_{-r}\rho _{1}(x,t)dx &=\int^{}_{\mathbb{R}}\rho _{1}(x,t)dx-\left( \int^{\infty}_{r}\rho _{1}(x,t)dx+\int^{-r}_{-\infty}\rho _{1}(x,t)dx  \right)\\
& \geq- \frac{\sigma}{4} +\int^{}_{\mathbb{R}}\rho_{1,0}(x)dx-\left( \int^{\infty}_{r}\widetilde{\rho} _{1}(x,t)dx+\int^{-r}_{-\infty}\widetilde{\rho} _{1}(x,t)dx  \right)\\
& = \int^{r}_{-r}\widetilde{\rho }_{1}(x,t)dx -\frac{\sigma}{4}.
\end{split}
\end{equation*}
Hence, \eqref{eq:FKvsChemotaxis_2} is obtained, and the proof is finished once we combine it with \eqref{eq:mass_rho_rho_tilde}.
\end{proof}

\section{Transport estimates based on comparison principles}
In this section. we show that a significant portion of the initial mass of $\rho_0$ gets transported inside $[- \frac{1}{2},\frac{1}{2}]$ under the potential $H$ in time $\tau\sim\frac{L}{\gamma}$.
Let us define the operator $L$ by $L\rho=-\Delta \rho +\nabla \cdot (\rho\nabla H)$, and we consider its dual operator $L^*f=-\Delta f-\nabla H\cdot\nabla f$, and consider the evolution
\begin{equation} \label{eq:equivalent_equation}
\partial_tf-\Delta f-\nabla H\cdot \nabla f=0.
\end{equation}
Here we will prove a duality result:
\begin{lemma}\label{prop:duality}
Let $\rho$ and $f$ be solutions to \eqref{eq:fokker_planck} and \eqref{eq:equivalent_equation} respectively, with initial conditions $\rho_0$ and $f_0$, where $\rho_0\in L^\infty(\R,e^{-H}dx)\cap L^1(\mathbb{R})$ and $f_0\in L^{\infty}(\mathbb{R})\cap L^1(\R,e^H dx)$. Then, for any $t>0$ and $s\in[0,t]$, the integral
\begin{equation}
    \int_{\mathbb{R}} \rho(x,s)f(x,t-s)dx
\end{equation}
does not depend on $s$.
\end{lemma}
\begin{proof}
By standard approximation arguments, it suffices to show the result for smooth, sufficiently decaying $\rho$, $f$. Taking the derivative in $s$ gives
\begin{align}
\begin{split}\label{eq:duality}
    \partial_s \int_{\mathbb{R}}  \rho(x,s)f(x,t-s)dx & = \int_{\mathbb{R}}  f(x,t-s)\partial_s \rho(x,s)dx+\rho(x,s)\partial_sf(x,t-s)dx\\
    &=\int_{\mathbb{R}}f(x,t-s)(\Delta \rho(x,s)-\nabla\cdot(\rho(x,s)\nabla H(x,s)))dx \\
    &-\int_{\mathbb{R}}
    \rho(x,s)(\Delta f(x,t-s)+\nabla f(x,t-s)\cdot\nabla H(x,t-s)))dx.
\end{split}\end{align}
Applying divergence theorem shows that the right hand side equals zero. 
\end{proof}
Let us define $r_1 = \frac{6}{25} < \frac14.$ The reason we need this auxiliary number is that our initial data for $\rho_2$ is slightly less than $\chi_{[-1/2,1/2]}(x),$ so to ensure that at least one quarter of the $\rho_2$ mass 
will get consumed we need to show transport of $\rho_1$ to an interval smaller than $[-1/4,1/4]$ - since the reaction will happen on the passage (see Section \ref{passage}). 
Our next step is the following proposition. 
\begin{prop}\label{thm:control_f}
Let $f(x,t)$ solve \eqref{eq:equivalent_equation} with $H$ given by \eqref{eq:least_concentrated_H}. Suppose initial data $f_0\in C_0^\infty$ is even with $0\leq f_0\leq 1$. $f_0$ is non-increasing in the even direction and $f_0(x)\geq \chi_{[-d_1,d_1]}$ where $d_1<r_1$. Then there exists a constant $C>0$ such that for all $t\geq 0$,
\begin{equation}
    f(x,t)\geq \frac{1}{C}\chi_{[-\frac{1}{C}({1+\gamma t}),\frac{1}{C}({1+\gamma t})]}.
\end{equation}
\end{prop}
\begin{proof}
Let us first fix $d_1$ and $d_0$ such that $\frac16 < d_0<d_1<r_1<\frac14$. For simplicity, we specify $d_0=\frac{11}{60}$ and $d_1=\frac{13}{30}$. One can pick other numbers, which will change the constant $C$. Now recall $H$ is even, and $f_0$ since is even as well, the solution $f$ to \eqref{eq:equivalent_equation} remains even for all times. Due to parabolic comparison principles, $f(x,t)$ is also non-increasing in the even direction, and satisfies $1\geq f_0(x,t)\geq 0$ for all times. Since both $f$ and $H$ are even, we will just show calculation on the positive real line.
First, observe that
\begin{align*}
\begin{split}
        \int_{d_0}^{d_1} e^H dx &= \int_{d_0}^\infty \chi_{[0,d_1]}e^H dx\leq \int_{d_0}^\infty f_0e^H dx\leq \int_0^\infty f_0e^H dx-\int_0^{d_0} f e^H dx \\
        &\leq \int_{d_0}^\infty fe^H dx\leq f(d_0,t) \int_{d_0}^\infty e^H dx.
\end{split}
\end{align*}
Here in the penultimate step we used that $\int_{\R} f e^H\,dx$ is conserved by evolution. 
From the above inequality, we obtain that
\begin{equation*}
    f(d_0,t)\geq\frac{\int_{d_0}^{d_1} e^H dx}{\int_{d_0}^\infty e^H dx}.
\end{equation*}
Due to monotonicity of $H$ on $[\frac16, \infty)$ we have
\begin{equation}\label{eq:inside}
   I_1:= \int_{d_0}^{d_1} e^{-\frac{\gamma}{24}(3-4x+ 12x^2)} dx \geq \frac{d_1-d_0}{2} e^{H(\frac{d_0+d_1}{2})} \gtrsim e^{-\frac{67 \gamma}{600}},
\end{equation}
while
\begin{equation}\label{eq:middle}
   I_2:= \int_{d_1}^{1/2} e^{-\frac{\gamma}{24}(3-4x+ 12x^2)} dx \lesssim (\frac12 -d_1) e^{H(d_1)} \lesssim e^{-\frac{809 \gamma}{7200}} 
\end{equation}
and 
\begin{equation}\label{eq:outside}
   I_3:= \int_{1/2}^{\infty} e^{-\frac{\gamma}{3}x} dx = \frac{3e^{-\gamma/6}}{\gamma}.
\end{equation}
Now observe that
\begin{align*}
    f(d_0,t) = \frac{\int_{d_0}^{d_1} e^H dx}{\int_{d_0}^\infty e^H dx}=\frac{I_1}{I_1+I_2+I_3}.
\end{align*}
Due to \eqref{eq:inside}, \eqref{eq:middle} and \eqref{eq:outside}, it is clear that as $\gamma$ increases, $\frac{\int_{d_0}^{d_1} e^H dx}{\int_{d_0}^\infty e^H dx}$ converges to 1. Therefore, for all sufficiently large $\gamma$ we have $f(d_0,t)\geq\frac{1}{2}$ for all times.

As before, we denote $r=|x|.$ 
Let $\omega\in C^2$ be an even function that is convex on $[d_0,\infty)$ and satisfies $\omega(d_0)=\frac{1}{2}$, $\omega(r)>0$ for $r\in [d_0,d_1]$, and $\omega(r)=0$ if $r\in(d_1,\infty)$. 
For $\phi\in[0,1]$, $r \geq d_0,$ define $\omega_{\phi}(r)=\omega(d_0+(r-d_0)\phi)$. 
Note that for $r \geq d_0,$ 
\begin{equation*}
    L^{*} \omega_\phi = \partial^2_x \omega_{\phi}(x)+\partial_x H (x) \partial_x \omega_{\phi}(x)\geq \frac{\gamma}{60}|\partial_x(\omega_\phi(x))| = \frac{\gamma}{60}|\partial_x \omega(d_0+(r-d_0)\phi)|\phi,
\end{equation*}
where we used that $\partial_x^2 \omega_\phi (x)\geq 0$, $\partial_x \omega_\phi (x)< 0,$ and a straightforward bound on $\partial_x H$ for $d_0 \leq r \leq d_1.$ 
Choose a decreasing $\phi(t)$ with $\phi(0)=1$. Define $F(x,t)=\omega_{\phi(t)}(x)$. Since we always have $f(d_0,t)\geq \frac{1}{2}=F(d_0,t)$ and $f_0\geq \chi_{[-d_1,d_1]}\geq \omega (|x|)$, we can be sure that $f(x,t)\geq F(x,t)$ in $\mathbb{R} \setminus  B_{d_0}$ for all times if $\partial_t F\leq L^*F$. We have \[ \partial_t F=(r-d_0)\partial_x \omega(d_0+(r-d_0)\phi)\phi^{'}(t) \] 
and $\partial_t F=L^{*}F=0$ if $d_0+(r-d_0)\phi\geq\frac{1}{2}$. Hence we need to ensure that
\begin{equation*}
    -(r-d_0)\phi^{'}(t)\leq \frac{\gamma}{60}\phi
\end{equation*}
when $r\leq d_0+\frac{\frac12-d_0}{\phi}.$ This would follows from 
\begin{equation}
    -\frac{\frac12-d_0}{\phi}\phi'(t)=-\frac{19}{60\phi}\phi^{'}(t)\leq \frac{\gamma}{60}\phi
\end{equation}
which holds provided that $\partial_t\left(\frac{1}{\phi(t)}\right)\leq \frac{1}{C}\gamma$. Hence, we can take
\begin{equation}
    \phi(t)=\frac{1}{1+\frac{\gamma t}{C}}.
\end{equation}
Now fix a constant $a<d_1-d_0$, then we can make
\begin{equation}
    d_0+\frac{1}{1+\frac{\gamma t}{C}}(r-d_0)\leq d_0+a
\end{equation}
for $d_0 \leq x \leq d_0+\frac{1}{C}(1+\gamma t)$ if we choose $C$ large enough. Then, we obtain
\begin{equation}
    f(x,t)\geq \omega
    \left(d_0+\frac{r-d_0}{1+\frac{\gamma t}{C}}\right)\geq \omega(d_0+a)\geq \frac{1}{C}>0,
\end{equation}
where we may have to adjust our constant $C$ to make it larger if necessary.
Note that due to monotonicity of $f(x,t)$ we have $f(x,t) \geq 1/2$ for $|x| \leq d_0.$  
\end{proof}
Here we give a corollary that states a fixed fraction of $\rho(x,t)$ satisfying \eqref{eq:fokker1} must enter $[-\frac{6}{25},\frac{6}{25}]$ in time $\lesssim L/\gamma.$ 
\begin{corollary}\label{corollary:enter_center}
Let $\rho(x,t)$ solve \eqref{eq:fokker1} with a potential $H$ given by \eqref{eq:least_concentrated_H}. Suppose that the initial data $\rho_0(x)\geq 0$ and $\int_{|x| \leq L} \rho_0(x)=M_0$. Then for all sufficiently large $\gamma$, there exists a constant $C$ such that if $t\geq T:=\frac{CL}{\gamma}$, we have
\begin{equation}
    \int_{-\frac{6}{25}}^{\frac{6}{25}}\rho(x,t)\geq \frac{ M_0}{C}.
\end{equation}
\end{corollary}
\begin{remark}
For simplicity, we picked a fixed range $[-\frac{6}{25},\frac{6}{25}]$. One can obtain similar result for any range that strictly contains $[-\frac{1}{6},\frac{1}{6}]$, but then all constants and the range of validity in $\gamma$ will depend on the range choice.
\end{remark}
\begin{proof}
Recall we chose $d_0=\frac{11}{60}$ and $d_1=\frac{13}{60}$, so we have $d_0<d_1<\frac{6}{25}$. Take $f_0\in C^\infty_0([-\frac{6}{25},\frac{6}{25}])$ as in Theorem \ref{thm:control_f}. Due to Proposition \ref{prop:duality}, we have
\begin{equation}
    \int_\mathbb{R}f_0(x)\rho(x,t) dx=\int_\mathbb{R}f(x,t)\rho_0(x)dx.
\end{equation}
Therefore, applying Theorem \ref{thm:control_f} we find that if $C$ is sufficiently large, then
\begin{equation}
    \int_{-\frac{6}{25}}^{\frac{6}{25}}\rho(x,t)dx \geq\int_\mathbb{R}f_0(x)\rho(x,t)dx=\int_\mathbb{R}f(x,t)\rho_0(x)dx\geq \frac{ M_0}{C}
\end{equation}
for $t \geq \frac{CL}{\gamma}.$ 
\end{proof}

\section{Decay for Density $\rho_2$ Based on a ``Pass-Through" Argument}\label{passage}
Let use now consider the original equations:
\begin{equation}\label{eq:org_sys}
\begin{cases}
\partial_t\rho _{1}-\Delta \rho _{1}+ \nabla \cdot \left(\rho _{1}\nabla \left( -\Delta \right) ^{-1}\rho _{2}\right)=-\varepsilon \rho _{1}\rho _{2}\\
\partial_t\rho _{1}=-\varepsilon \rho _{1}\rho _{2}.
\end{cases}
\end{equation}
In this section, we obtain decay estimates for $\rho_2$ and complete the proof of Theorem \ref{thm:main}. Let us first discuss the idea of the proof. 
We will obtain that $T_C \leq T$ by contradiction, where $T=\frac{CL}{\gamma}$ is the time defined in Corollary \ref{corollary:enter_center}. Assume $T<T_C$, that is, for all $t \leq T$ the total mass of $\rho_2$ remains greater than $3/4$ of the original mass. Then combining Proposition  \ref{thm:FKvsChemotaxis} and Corollary \ref{corollary:enter_center} gives us that
\begin{equation}\label{conc319}
    \int_{-\frac{6}{25}}^{\frac{6}{25}} \rho_1(x,t)dx \geq \frac{M_0}{C}-\frac{\sigma}{4},
\end{equation}
Since we assume that $\sigma \ll M_0$, this implies that
 a significant portion of $\rho_1$ has entered the range $[-\frac{6}{25},\frac{6}{25}]$ by time $T$.
During the process of moving inside this range, $\rho_1$ has to react with $\rho_2$. 
Since the drift velocity $\partial_r(-\Delta)^{-1} \rho_2\sim-\gamma$ for all $r\in(\frac{6}{25},\frac{1}{2})$ and $t\leq T$, heuristically, each particle of $\rho_1$ takes $\frac{1}{\gamma}$ time to pass through the region $[\frac{6}{25},\frac{1}{2}]$, and interacts with $\rho_2$ during  this time with interaction coefficient $\epsilon$. 
If $\frac{\epsilon M_0}{\gamma}\gg 1$, most of $\rho_2$ mass in $[\frac{6}{25}, \frac12]$ and thus a quarter of the initial mass of $\rho_2$ will react with $\rho_1$ by time $T$, contradicting our assumption. 
Therefore, we obtain $T_C\leq T$ as conclusion.

The goal of this section is to rigorously justify the above heuristics. The key step is the following proposition.
\begin{prop}\label{prop:prop_goal}
Let $(\rho_1,\rho_2)$ be a solutions to \eqref{eq:org_sys} with the initial data satisfying \eqref{rho1init}, \eqref{rho2init}, and let the parameters of the problem satisfy 
\eqref{assum1222} with sufficiently large $B.$  
Then there exists a constant $C$ (that may only depend on $B$) such that if $T = \frac{CL}{\gamma}$ then 
\begin{align}\label{eqn:prop_goal}
\int_0^{T}(\rho_1(x,t)+\rho_1(-x,t))dt\geq \frac{ M_0}{C\gamma}\text{, for all }x\in\left(\frac{6}{25},\frac{1}{2}\right).
\end{align}
\end{prop}
Note that from \eqref{eqn:prop_goal}, we can say that either
\begin{equation}
    \int_0^{T}\rho_1(x,t)dt\geq \frac{ M_0}{2C\gamma},
\end{equation}
or
\begin{equation}
    \int_0^{T}\rho_1(-x,t)dt\geq \frac{ M_0}{2C\gamma}.
\end{equation}
Before we prove the proposition, let us point out an immediate consequence of it. 
Since $\partial_t\rho_2=-\epsilon\rho_1\rho_2$, we have $\rho_2(x,t)=\rho_2(x,0)e^{-\epsilon\int_0^t\rho_1(x,s)ds}$ for every $x$, thus the proposition directly leads to the following:
\begin{prop}\label{prop:prop_goal_2}
Under the assumptions of Proposition \ref{prop:prop_goal}, for every $x\in(\frac{6}{25},\frac{1}{2})$, at least one of the following must hold:
\begin{align}
\begin{split} \label{eq:decay_rho2}
\frac{\rho_2(x,T)}{\rho_2(x,0)}&=\exp\left\{-\epsilon\int_0^{T}\rho_1(x,t)dt\right\}\leq \exp\left\{-\frac{\epsilon M_0}{2C\gamma}\right\},\\
\frac{\rho_2(-x,T)}{\rho_2(-x,0)}&=\exp\left\{-\epsilon\int_0^{T}\rho_1(-x,t)dt\right\}\leq \exp\left\{-\frac{\epsilon M_0}{2C\gamma}\right\}.
\end{split}
\end{align}
\end{prop}
Therefore if $\frac{\epsilon M_0}{\gamma}\gg 1$, then most of the mass $\rho_2$ originally situated in a set of measure at least $\frac{13}{50}$ 
would be gone by time $T$. We can make the share of remaining $\rho_2$ in this set as close to zero as we want by adjusting the constant $B.$ 
Since $\frac{13}{50}>\frac14,$ we get that $T_C \leq T,$ completing the proof of Theorem \ref{thm:main}. 
\begin{proof}[Proof of Proposition \ref{prop:prop_goal}]
We organize our proof into three steps:
\newline \textbf{Step 1}: We show an average version of estimate \eqref{eqn:prop_goal}.
To this end, we define the quantity $M(r,t):=\int_{-r}^{r}\rho_1(x,t)\,dx$ and  define $\tilde H(\cdot, t)=\chi(-\Delta)^{-1}\rho_2(\cdot,t).$ 
Direct calculation yields that
\begin{align}\label{eq:equation_for_M}
\partial_t M = \partial_r^2 M - \rho_1(r,t) \partial_r \tilde H(r,t) - \rho_1(-r,t) \partial_r \tilde H(-r,t)-\epsilon\int_{-r}^r\rho_1(x,t)\rho_2(x,t)dx 
\end{align}
Note that $\partial_r M(r,t)=\rho_1(r,t)+\rho_1(-r,t)$, so it suffices to show that 
\begin{equation}\label{eq:goal_prop_new}
\int_{0}^{T} \partial_r M(r,t)dt\geq\frac{M_0}{C\gamma}\text{, for all }r\in(\frac{6}{25},\frac{1}{2}).
\end{equation}
Let $I=(a,b)\subset(\frac{6}{25},\frac{1}{2})$ be an arbitrary interval. For any $s\in(a,b)$, integrating \eqref{eq:equation_for_M} over $(a,s)$ gives
\begin{equation*}
\begin{split}
\int_a^s \partial_t M(r,t)dr &\leq partial_r M(s,t)-\partial_r M(a,t)+ \int_a^s\left| \partial_r M(r,t)(\partial_r \tilde{H}(r,t) + \partial_r (\tilde{H}(-r,t))\right|dr\\
&\leq \partial_r M(s,t)+C\gamma\int_a^s \partial_r M(r,t)\,dr,
\end{split}
\end{equation*}
where in the inequality we used that $\rho_1, \rho_2 \geq 0$ and $\partial \tilde{H}(\pm r,t)\geq -C\gamma,\ \forall r\in(\frac{6}{25},\frac{1}{2})$. 
Integrating the above inequality from $t=0$ to $t=T$ gives
\begin{equation*}
\int_a^s(M(r,T)-M(r,0))dr\leq\int_0^{T} \partial_r M(s,t)\,dt+C\gamma\int_0^{T}\int_a^s \partial_r M(r,t)\,drdt.
\end{equation*}
The left hand side can be estimated as follows. For any $r\in(\frac{6}{25},\frac{1}{2})$, we have $M(r,0) \ll 1$ due to \eqref{rho1init}. 
Combining Proposition  \ref{thm:FKvsChemotaxis} and Corollary \ref{corollary:enter_center} (see \eqref{conc319}), we have $M(r,T)\geq \frac{M_0}{C}$ for some constant $C$ and for all $r\in(\frac{6}{25},\frac{1}{2})$. Thus $M(r,T)-M(r,0)\geq\frac{M_0}{C_1}$ for all $r\in(a,s)\subset(\frac{6}{25},\frac{1}{2})$, and the above inequality becomes
\begin{equation*}
\int_0^{T}\partial_r M(s,t)dt+C\gamma\int_0^{T}\int_a^s \partial_r M(r,t)drdt\geq\frac{(s-a)M_0}{C_1}.
\end{equation*}
We then integrate this inequality over all $s\in I=(a,b)$, and obtain that
\begin{equation*}
\int_0^{T}\int_{a}^b \partial_r M (s,t )dsdt+C\gamma\int_0^{T}(b-a)\int_a^b \partial_r M(r,t)drdt\geq\frac{(b-a)^2M_0}{2C_1},
\end{equation*}
that is,
\begin{equation*}
\int_0^{T}\frac{1}{|I|}\int_I \partial_r M(s,t)dsdt\geq\frac{M_0}{2C_1(|I|^{-1}+C\gamma)}.
\end{equation*}
Therefore, for any interval $I\subset(\frac{6}{25},\frac{1}{2})$ with $|I|=\gamma^{-1}$, we have 
\begin{equation}\label{eq:int_average}
\frac{1}{|I|}\int_I\int_0^{T} \partial_r M(s,t)dtds\geq\frac{M_0}{C\gamma}.
\end{equation}
 This inequality shows that \eqref{eqn:prop_goal} holds in the interval $I$ in an average sense. This concludes Step 1.\\
 \textbf{Step 2:} We improve the average bound \eqref{eq:int_average} to a local point-wise control. Here we rule out the possibility that $\int_0^{T} \partial_r M(s,t)dt$ is distributed very non-uniformly among $s\in I$. Since $\rho_1(x,t)$ is a solution to the parabolic partial differential equation
 \begin{equation}\label{eq:Yt}
 \partial_t \rho_1-\partial_{x}^2 \rho_{1}+\partial_x \tilde H\partial_x\rho_1+(\partial_{x}^2 \tilde H+\varepsilon \rho_2)\rho_1= 0,
 \end{equation}
 such situation will not happen.
  Let $\rho_1$ be a non-negative solution of \eqref{eq:Yt}, and re-scale \eqref{eq:Yt} by setting $y=\gamma x$, $\tau=\gamma^2(t-t_0)$. In the new coordinates, $\rho_1^*(y,\tau)=\rho_1(x,t)$ satisfies
 \begin{equation*}
 \partial_\tau \rho_1^* - \partial_{y}^2\rho_1^* +b(y)\partial_y\rho_1^*+c(y,\tau)\rho_1^*=0,
 \end{equation*}
 where $|b(y)|\leq C$ and $|c(y,\tau)|\leq C$ for all $y$ and $\tau\geq 0.$ Here the bounds on $b$ and $c$ follow from the facts that $|\tilde{H}_x|\leq C\gamma, |\partial^2_x \tilde H|\leq C\gamma,\rho_2\leq ||\rho_2(\cdot,0)||_\infty\leq \sigma$, and 
 $\varepsilon \sigma \gamma^{-2} = \varepsilon \chi^{-2}\sigma^{-1} \leq a^{-1}$ due to \eqref{assump319}.
 Now by the parabolic Harnack inequality (see e.g. \cite{Lieberman}), for any interval $I'\subset(\frac{6}{25}\gamma,\frac{1}{2}\gamma)$ with length 1, we get that
 \begin{equation*}
 \rho_1^*(y,\tau)\geq \frac{1}{C}\int_{I'}\int_0^1\rho_1^* (y,\tau)d\tau dy\text{, for all }\pm y\in I',\tau\in[1,2].
 \end{equation*}
 Translating this back into original coordinates, there exists a universal constant $C>0$ such that the following estimate holds
 \begin{equation}\label{eq:harnack_local}
 \rho_1(x,t)\geq
\frac{\gamma^3}{C}\int_I\int_{t_0}^{t_0+\gamma^{-2}}\rho_1(x,t)dtdx\text{, for all }\pm x\in I, t\in[t_0+\gamma^{-2},t_0+2\gamma^{-2}],
 \end{equation}
for any interval $I\subset(\frac{6}{25},\frac{1}{2})$ with $|I|=\gamma^{-1}$, and $t_0\geq 0$. This concludes Step 2.
\newline
\textbf{Step 3:} Now we improve the local control \eqref{eq:harnack_local} to the global lower bound stated in Proposition \ref{prop:prop_goal}.
Let us define the time intervals $J_k:=[k\gamma^{-2},(k+1)\gamma^{-2}]$ for $k\in\mathbb{N} \cup\{0\}$. 
-+Let $n$ be the smallest integer such that $(n+1)\gamma^{-2}\geq T$, thus $[0,T]\subset\cup_{0\leq k\leq n} J_k$. Then for any interval $I\subset(\frac{6}{25},\frac{1}{2})$ with $|I|=\gamma^{-1}$, we can rewrite \eqref{eq:int_average} as
 \begin{equation}\label{eq:int_average2}
 \sum_{k=0}^n\int_{J_k}\int_I M_r(r,t)drdt\geq \frac{ M_0}{C\gamma^2}.
 \end{equation}
 Meanwhile, the local estimate \eqref{eq:harnack_local} and the fact that $\rho_1(r,t)+\rho_1(-r,t)= \partial_r M(r)$ yield that
 \begin{equation}\label{eq:m_jk}
 \inf_{r\in I,t\in J_{k+1}}\partial_r M(r,t)\geq
 \frac{\gamma^3}{C}\int_{J_k}\int_I \partial_r M(r,t)drdt\text{, for }k\geq 0.
 \end{equation}
 We then sum up this inequality over $0\leq k\leq n$, and combine it with \eqref{eq:int_average2} to obtain
 \begin{equation*}
 \inf_{r\in I}\int_0^{(n+2)\gamma^{-2}}\partial_r M(r,t)dt\geq \inf_{r\in I}\sum_{k=0}^{n+1}\gamma^{-2}\inf_{t\in J_k} \partial_r M(r,t)dt\geq\frac{M_0}{C\gamma}\text{, for }r\in I.
 \end{equation*}
 Observe that since $L \geq 1$ and $\gamma \geq 1,$ we have $\gamma^{-2} \leq L/\gamma.$ Thus by adjusting slightly the constant $C$, we can take $T \sim L/\gamma$ such that 
 $(n+2)\gamma^{-2}\leq T.$ 
 Thus
 \begin{equation}
 \int_0^{T} \partial_r M(r,t)dt\geq\frac{M_0}{C\gamma}\text{, for all }r\in I,
 \end{equation}
 and since $I\subset(\frac{6}{25},\frac{1}{2})$ is an arbitrary interval with length $\gamma^{-1}$, the above inequality holds for all $r\in(\frac{6}{25},\frac{1}{2})$, which finishes the proof of Proposition \ref{prop:prop_goal}.
 \end{proof}
As explained above after Proposition \ref{prop:prop_goal_2}, this also completes the proof of Theorem \ref{thm:main}. 

Finally, we do not consider the case where $M_0 \eps/\gamma$ is small in this paper exactly because in this case we cannot establish that significant reaction happens ``on the passage". 
On the other hand, since the drift $\sim \gamma$ can be quite large, it can lead to very non-uniform distribution of $\rho_1$ once it moves onto support of $\rho_2.$ This can slow down  
the reaction significantly. Such slowdown, however, is artificial: in realistic biological systems the drift magnitude is rarely large since the speed of the agents has limitations. 
Thus a classical Keller-Segel form of chemotaxis appears ill-suited for modelling the small $M_0 \eps/\gamma$ case. Instead a flux-limited version of chemotactic drift seems more appropriate - 
see \cite{KNRY} for more details. 

\noindent {\bf Acknowledgement.} \rm The authors acknowledge partial support of the NSF-DMS grants 1848790 and 2006372. AK has also been partially supported by Simons Foundation. 

\end{document}